\documentclass[12pt]{article}
\usepackage{amsfonts}
\usepackage{amsmath,amscd,amsthm,a4,amssymb}
\usepackage{amsmath,amscd,amsthm,a4,amssymb}

\newtheorem{theorem}{Theorem}
{}
\newtheorem{corollary}{Corollary}
{}

{}

{}

{}
\theoremstyle{plain}
{}

{}

{}

{}

{}

{}

{}

{}

{}

{}

{}

{}

{}

{}

{}

{}

\begin{document}
\begin{center}
{\Large \bf{Boost Invariant Surfaces with Pointwise 1-Type Gauss Map in
Minkowski 4-Space E$_{1}^{4}$ }}
\end{center}
\centerline{\large Ferdag KAHRAMAN AKSOYAK  $^{1}$, Yusuf YAYLI $^{2}${\footnotetext{
{E-mail: $^{1}$ferda@erciyes.edu.tr(F. Kahraman Aksoyak ); $^{2}$yayli@science.ankara.edu.tr (Y.Yayli)}} }}

\

\centerline{\it $^{1}$Erciyes University, Department of Mathematics,
Kayseri, Turkey}
\centerline{\it $^{2}$Ankara University, Department of Mathematics,
Ankara, Turkey}

\begin{abstract}
In this paper, we study spacelike rotational surfaces which are called boost
invariant surfaces in Minkowski 4-space $\mathbb{E}_{1}^{4}$. We give
necessary and sufficient condition for flat spacelike rotational surface to
have pointwise 1-type Gauss map. Also, we obtain a characterization for
boost invariant marginally trapped surface with pointwise 1-type Gauss map.
\end{abstract}

\begin{quote}\small
\end{quote}
\begin{quote}\small
\end{quote}

\section{Introduction}

The notion of finite type mapping was introduced by B.Y. Chen in late
1970's. A pseudo- Riemannian submanifold $M$ of the $m-$dimensional
pseudo-Euclidean space $\mathbb{E}_{s}^{m}$ is said to be of finite type if
its position vector $x$ can be expressed as a finite sum of eigenvectors of
the Laplacian $\Delta $ of $M$, that is, $x=x_{0}+x_{1}+...+x_{k}$, where $%
x_{0}$ is a constant map, $x_{1},...,x_{k}$ are non-constant maps such that $%
\Delta x_{i}=\lambda _{i}x_{i},$ $\lambda _{i}\in $ $\mathbb{R}$, $%
i=1,2,...,k.$ If $\lambda _{1},\lambda _{2},$...,$\lambda _{k}$ are all
different, then $M$ is said to be of $k-$type. This notion of finite type
immersions is naturally extended to differentiable maps of $M$ in
particular, to Gauss maps of submanifolds \cite{chen1}.

If a submanifold $M$ of a Euclidean space or pseudo-Euclidean space has
1-type Gauss map $G$, then $G$ satisfies $\Delta G=\lambda \left( G+C\right)
$ for some $\lambda \in \mathbb{R}$ and some constant vector $C.$ Chen and
Piccinni made a general study on compact submanifolds of Euclidean spaces
with finite type Gauss map and they proved that a compact hypersurface $M$
of $\mathbb{E}^{n+1}$ has 1-type Gauss map if and only if $M$ is a
hypersphere in $\mathbb{E}^{n+1}$ \cite{chen1}$.$

Hovewer the Laplacian of the Gauss map of some typical well-known surfaces
such as\ a helicoid, a catenoid and right cone in 3-dimensional Euclidean
space $E^{3}$ and a helicoids of the 1st,2nd and 3rd kind, conjugate
Enneper's surface of the second kind \ and B-scrolls in 3- dimensional
Minkowski space $E_{1}^{3}$ take a somewhat different form namely,
\begin{equation}
\Delta G=f\left( G+C\right)
\end{equation}%
for some non-zero smooth function $f$ on $M$ and some constant vector $C.$
This equation is similar to an eigenvalue problem but the smooth function $f$
is not always constant. So a submanifold $M$ of a pseudo-Euclidean space $%
\mathbb{E}_{s}^{m}$ is said to have pointwise 1-type Gauss map if its Gauss
map satisfies $\left( 1\right) $ for some smooth function $f$ on $M$ and
some constant vector $C.$ A submanifold with pointwise 1-type Gauss map is
said to be of the first kind if the vector $C$ in $\left( 1\right) $ is
zero\ vector. Otherwise, the pointwise 1-type Gauss map is said to be of the
second kind.

Surfaces in Euclidean space and in pseudo-Euclidean space with pointwise
1-type Gauss map were recently studied in \cite{chen}, \cite{choi1}, \cite%
{choi2}, \cite{choi3}, \cite{dursun2}, \cite{dursun3}, \cite{dursun4}, \cite%
{dursun5}, \cite{kim2}, \cite{niang1}, \cite{niang2}. Also Dursun and Turgay
in \cite{dursun1} gave all general rotational surfaces in $\mathbb{E}^{4}$
with proper pointwise 1-type Gauss map of the first kind and classified
minimal rotational surfaces with proper pointwise 1-type Gauss map of the
second kind. Arslan et al. in \cite{arslan1} investigated rotational
embedded surface with pointwise 1-type Gauss map. Arslan at el. in \cite%
{arslan2} gave necessary and sufficent conditions for Vranceanu rotation
surface to have pointwise 1-type Gauss map. Yoon in \cite{yoon2} showed that
flat Vranceanu rotation surface with pointwise 1-type Gauss map is a
Clifford torus and in \cite{yoon1} studied rotation surfaces in the
4-dimensional Euclidean space with finite type Gauss map. Kim and Yoon in
\cite{kim1} obtained the complete classification theorems for the flat
rotation surfaces with finite type Gauss map and pointwise 1-type Gauss map.
The authors in \cite{ak1} studied flat general rotational surfaces in the 4-
dimensional Euclidean space $\mathbb{E}^{4}$ with pointwise 1-type Gauss map
and they showed that a non-planar flat general rotational surfaces with
pointwise 1-type Gauss map is a Lie group if and only if it is a Clifford
Torus. Also they gave a characterization for flat general rotation surfaces
with pointwise 1-type Gauss map in the 4- dimensional pseudo-Euclidean space
$\mathbb{E}_{2}^{4}$ \cite{ak2}.

On the other hand, trapped surfaces, introduced by Penrose in 1965, have a
fundamental role in the study of the singularity theorems in General
Relativity. If the mean curvature vector of a surface in $E_{1}^{4}$ is
timelike everywhere, \i t is called trapped surfaces; if the mean curvature
vector is always null (the mean curvature vector is proportional to one of
the null normals), the surface is called marginally trapped surface. Since
the mean curvature of such spacelike surface $H$ satisfy $\left \Vert
H\right \Vert =0,$ in mathematical literature these surfaces are called
quasi-minimal. In general relativity, marginally trapped surfaces are used
the study of the surfaces of black hole.

S.Haesen and M. Ortega in \cite{ortega1} and \cite{ortega2} classified
marginally trapped surfaces which are invariant under a spacelike rotations
and boost transformations in Minkowski 4-space. Also B. Y. Chen classify
marginally trapped Lorentzian flat surfaces and biharmonic surfaces in the
Pseudo Euclidean space $E_{2}^{4}$ \cite{chen2}$.$ Milousheva in \cite{milo}
studied marginally trapped surface with pointwise 1-type Gauss map in
Minkowski 4-space and proved that marginally trapped surface is of pointwise
1-type Gauss map if and only if it has parallel mean curvature vector field.

In this paper, we study spacelike surfaces which are invariant under boost
transformation (hyperbolic rotations) in Minkowski 4-space. We give a
characterization of flat spacelike rotational surface with pointwise 1-type
Gauss map. Also we obtain a characterization for boost invariant marginally
trapped surface with pointwise 1-type Gauss map and give an example of such
surfaces.

\section{Preliminaries}

Let $E_{s}^{m}$ be the $m-$dimensional pseudo-Euclidean space with signature
$(s,m-s)$. Then the metric tensor $g$ in $E_{s}^{m}$ has the form
\begin{equation*}
g=\sum \limits_{i=1}^{m-s}\left( dx_{i}\right) ^{2}-\sum
\limits_{i=m-s+1}^{m}\left( dx_{i}\right) ^{2}
\end{equation*}%
where $(x_{1},...,x_{m})$ is a standard rectangular coordinate system in $%
E_{s}^{m}.$

Let $M$ be an $n-$dimensional pseudo-Riemannian submanifold of a $m-$%
dimensional pseudo-Euclidean space $\mathbb{E}_{s}^{m}.$ We denote
Levi-Civita connections of $\mathbb{E}_{s}^{m}$ and $M$ by $\tilde{\nabla}$
and $\nabla ,$ respectively. Let $e_{1},$...,$e_{n},e_{n+1},$...,$e_{m}$ be
an adapted local orthonormal frame in $\mathbb{E}_{s}^{m}$ such that $e_{1},$%
...,$e_{n}$ are tangent to $M$\ and $e_{n+1},$...,$e_{m}$ normal to $M.$ We
use the following convention on the ranges of indices: $1\leq i,j,k,$...$%
\leq n$, $n+1\leq r,s,t,$...$\leq m$, $1\leq A,B,C,$...$\leq m.$

Let $\omega _{A}$ be the dual-1 form of $e_{A}$ defined by $\omega
_{A}\left( X\right) =\left \langle e_{A},X\right \rangle $ and $\varepsilon
_{A}=\left \langle e_{A},e_{A}\right \rangle =\pm 1.$ Also, the connection
forms $\omega _{AB}$ are defined by%
\begin{equation*}
de_{A}=\sum \limits_{B}\varepsilon _{B}\omega _{AB}e_{B},\text{ \  \ }\omega
_{AB}+\omega _{BA}=0
\end{equation*}%
Then we have
\begin{equation*}
\tilde{\nabla}_{e_{k}}^{e_{i}}=\sum \limits_{j=1}^{n}\varepsilon _{j}\omega
_{ij}\left( e_{k}\right) e_{j}+\sum \limits_{r=n+1}^{m}\varepsilon
_{r}h_{ik}^{r}e_{r}
\end{equation*}%
and%
\begin{equation}
\tilde{\nabla}_{e_{k}}^{e_{s}}=-\sum \limits_{j=1}^{n}\varepsilon
_{j}h_{kj}^{s}e_{j}+\sum \limits_{r=n+1}^{m}\varepsilon _{r}\omega
_{sr}\left( e_{k}\right) e_{r},\text{ \  \  \  \ }D_{e_{k}}^{e_{s}}=\sum%
\limits_{r=n+1}^{m}\omega _{sr}\left( e_{k}\right) e_{r},
\end{equation}%
where $D$ is the normal connection, $h_{ik}^{r}$ the coefficients of the
second fundamental form $h.$ The mean curvature vector $H$ of $M$ in $%
\mathbb{E}_{s}^{m}$ is defined by
\begin{equation*}
H=\frac{1}{n}\sum \limits_{s=n+1}^{m}\sum \limits_{i=1}^{n}\varepsilon
_{i}\varepsilon _{s}h_{ii}^{s}e_{s}
\end{equation*}%
and the Gaussian curvature $K$ of $M$ is given by
\begin{equation*}
K=\sum \limits_{s=n+1}^{m}\varepsilon _{s}\left(
h_{11}^{s}h_{22}^{s}-h_{12}^{s}h_{21}^{s}\right)
\end{equation*}%
Also normal curvature tensor $R^{D}$ of $M$ in $\mathbb{E}_{s}^{m=n+2}$ is
given by%
\begin{equation}
R^{D}(e_{j},e_{k};e_{r},e_{s})=\sum \limits_{i=1}^{n}\varepsilon _{i}\left(
h_{ik}^{r}h_{ij}^{s}-h_{ij}^{r}h_{ik}^{s}\right)
\end{equation}%
We recall that a surface $M$ in $\mathbb{E}_{1}^{4}$ is called extremal
surface if its mean curvature vector vanishes. If its Gaussian curvature
vanishes, the surface $M$ is called flat surface. If its normal curvature
tensor $R^{D}$ vanishes identically then a surface $M$ in $\mathbb{E}_{1}^{4}
$ is said to have flat normal bundle.

For any real function $f$ on $M$ the Laplacian $\Delta f$ of $f$ is given by
\begin{equation}
\Delta f=-\varepsilon _{i}\sum \limits_{i}\left( \tilde{\nabla}_{e_{i}}\tilde{\nabla}%
_{e_{i}}f-\tilde{\nabla}_{\nabla _{e_{i}}^{e_{i}}}f\right)
\end{equation}%
Let us now define the Gauss map $G$ of a submanifold $M$ into $G(n,m)$ in $%
\wedge ^{n}\mathbb{E}_{s}^{m},$ where $G(n,m)$ is the Grassmannian manifold
consisting of all oriented $n-$planes through the origin of $\mathbb{E}%
_{s}^{m}$ and $\wedge ^{n}\mathbb{E}_{s}^{m}$ is the vector space obtained
by the exterior product of $n$ vectors in $\mathbb{E}_{s}^{m}.$ Let $%
e_{i_{1}}\wedge ...\wedge e_{i_{n}}$ and $f_{j_{1}}\wedge ...\wedge
f_{j_{n}} $be two vectors of $\wedge ^{n}\mathbb{E}_{s}^{m},$ where $%
\left
\{ e_{1},\text{...,}e_{m}\right \} $ and $\left \{ f_{1},\text{...,}%
f_{m}\right \} $ are orthonormal bases of $\mathbb{E}_{s}^{m}$. Define an
indefinite inner product $\left \langle ,\right \rangle $ on $\wedge ^{n}%
\mathbb{E}_{s}^{m}$ by%
\begin{equation*}
\left \langle e_{i_{1}}\wedge ...\wedge e_{i_{n}},f_{j_{1}}\wedge ...\wedge
f_{j_{n}}\right \rangle =\det \left( \left \langle e_{i_{l}},f_{j_{k}}\right
\rangle \right) .
\end{equation*}%
Therefore, for some positive integer $t,$ we may identify $\wedge ^{n}%
\mathbb{E}_{s}^{m}$ with some Euclidean space $\mathbb{E}_{t}^{N}$ where $%
N=\left(
\begin{array}{c}
m \\
n%
\end{array}%
\right) .$ The map $G:M\rightarrow G(n,m)\subset E_{t}^{N}$ defined by $%
G(p)=\left( e_{1}\wedge ...\wedge e_{n}\right) \left( p\right) $ is called
the Gauss map of $M,$ that is, a smooth map which carries a point $p$ in $M$
into the oriented $n-$plane in $\mathbb{E}_{s}^{m}$ obtained from parallel
translation of the tangent space of $M$ at $p$ in $\mathbb{E}_{s}^{m}.$

\section{Boost Invariant Surfaces with Pointwise 1-Type Gauss Map in $%
E_{1}^{4}$}

In this section, we consider spacelike surfaces in the Minkowski space $%
E_{1}^{4}$ which are invariant under the following subgroup of direct,
linear isometries of $E_{1}^{4}$:
\begin{equation*}
G=\left \{
\begin{pmatrix}
\cos t & -\sin t & 0 & 0 \\
\sin t & \cos t & 0 & 0 \\
0 & 0 & 1 & 0 \\
0 & 0 & 0 & 1%
\end{pmatrix}%
:t\in \mathbb{R}\right \} ,
\end{equation*}%
well-known as boost isometries.%
\begin{equation*}
\varphi \left( t,s\right) =%
\begin{pmatrix}
\cos t & -\sin t & 0 & 0 \\
\sin t & \cos t & 0 & 0 \\
0 & 0 & 1 & 0 \\
0 & 0 & 0 & 1%
\end{pmatrix}%
\left(
\begin{array}{c}
\alpha _{1}(s) \\
0 \\
\alpha _{3}(s) \\
\alpha _{4}(s)%
\end{array}%
\right)
\end{equation*}%
\begin{equation}
M:\text{ }\varphi \left( t,s\right) =\left( \alpha _{1}(s)\cosh t,\alpha
_{1}(s)\sinh t,\alpha _{3}(s),\alpha _{4}(s)\right)
\end{equation}%
where the profile curve of $M$ is unit speed spacelike curve, that is, $%
-\left( \alpha _{1}^{\prime }(s)\right) ^{2}+\left( \alpha _{3}^{\prime
}(s)\right) ^{2}+\left( \alpha _{4}^{\prime }(s)\right) ^{2}=1.$ We choose a
moving frame $e_{1},e_{2},e_{3},e_{4}$ such that $e_{1},e_{2}$ are tangent
to $M$ and $e_{3},e_{4}$ are normal to $M$ which are given by the following:
\begin{eqnarray*}
e_{1} &=&\left( \alpha _{1}^{\prime }(s)\cosh t,\alpha _{1}^{\prime
}(s)\sinh t,\alpha _{3}^{\prime }\left( s\right) ,\alpha _{4}^{\prime
}\left( s\right) \right) \\
e_{2} &=&\left( \sinh t,\cosh t,0,0\right) \\
e_{3} &=&\frac{1}{\sqrt{1+\left( \alpha _{1}^{\prime }(s)\right) ^{2}}}%
((1+\left( \alpha _{1}^{\prime }(s)\right) ^{2})\cosh t,(1+\left( \alpha
_{1}^{\prime }(s)\right) ^{2})\sinh t \\
&&,\alpha _{1}^{\prime }(s)\alpha _{3}^{\prime }(s),\alpha _{1}^{\prime
}(s)\alpha _{4}^{\prime }(s) \\
e_{4} &=&\frac{1}{\sqrt{1+\left( \alpha _{1}^{\prime }(s)\right) ^{2}}}%
\left( 0,0,-\alpha _{4}^{\prime }(s),\alpha _{3}^{\prime }(s)\right)
\end{eqnarray*}%
Then it is easily seen that
\begin{equation*}
\left \langle e_{1},e_{1}\right \rangle =\left \langle e_{2},e_{2}\right
\rangle =\left \langle e_{4},e_{4}\right \rangle =1,\text{ }\left \langle
e_{3},e_{3}\right \rangle =-1
\end{equation*}%
we have the dual 1-forms as:
\begin{equation}
\omega _{1}=ds\text{ \  \  \  \ and \  \  \  \ }\omega _{2}=\alpha _{1}(s)dt
\end{equation}%
By a direct computation we have components of the second fundamental form
and the connection forms as:%
\begin{eqnarray}
h_{11}^{3} &=&-c(s),\ h_{12}^{3}=0,\ h_{22}^{3}=-b(s) \\
h_{11}^{4} &=&d(s),\text{ \ }h_{12}^{4}=0,\text{ \ }h_{22}^{4}=0  \notag
\end{eqnarray}%
\begin{eqnarray}
\omega _{12} &=&a(s)b(s)\omega _{2},\text{ \  \ }\omega _{13}=-c(s)\omega
_{1},\text{ \  \ }\omega _{14}=d(s)\omega _{1} \\
\omega _{23} &=&-b(s)\omega _{2},\text{ \  \ }\omega _{24}=0,\text{ \  \ }%
\omega _{34}=a(s)d(s)\omega _{1}  \notag
\end{eqnarray}%
By covariant differentiation with respect to $e_{1}$ and $e_{2}$ a
straightforward calculation gives:
\begin{eqnarray}
\tilde{\nabla}_{e_{1}}e_{1} &=&c(s)e_{3}+d(s)e_{4} \\
\tilde{\nabla}_{e_{2}}e_{1} &=&a(s)b(s)e_{2}  \notag \\
\tilde{\nabla}_{e_{1}}e_{2} &=&0  \notag \\
\tilde{\nabla}_{e_{2}}e_{2} &=&-a(s)b(s)e_{1}+b(s)e_{3}  \notag \\
\tilde{\nabla}_{e_{1}}e_{3} &=&c(s)e_{1}+a(s)d(s)e_{4}  \notag \\
\tilde{\nabla}_{e_{2}}e_{3} &=&b(s)e_{2}  \notag \\
\tilde{\nabla}_{e_{1}}e_{4} &=&-d(s)e_{1}+a(s)d(s)e_{3}  \notag \\
\tilde{\nabla}_{e_{2}}e_{4} &=&0  \notag
\end{eqnarray}%
where
\begin{equation}
a(s)=\frac{\alpha _{1}^{\prime }(s)}{\sqrt{1+\left( \alpha _{1}^{\prime
}(s)\right) ^{2}}}
\end{equation}%
\begin{equation}
b(s)=\frac{\sqrt{1+\left( \alpha _{1}^{\prime }(s)\right) ^{2}}}{\alpha
_{1}(s)}
\end{equation}%
\begin{equation}
c(s)=\frac{\alpha _{1}^{\prime \prime }(s)}{\sqrt{1+\left( \alpha
_{1}^{\prime }(s)\right) ^{2}}}
\end{equation}%
\begin{equation}
d(s)=\frac{-\alpha _{3}^{\prime \prime }(s)\alpha _{4}^{\prime }(s)+\alpha
_{4}^{\prime \prime }(s)\alpha _{3}^{\prime }(s)}{\sqrt{1+\left( \alpha
_{1}^{\prime }(s)\right) ^{2}}}
\end{equation}%
The Gaussian curvature $K$ of $M$ is given by
\begin{equation}
K=-b(s)c(s)
\end{equation}%
The mean curvature $H$ of $M$ is given by
\begin{equation}
H=\frac{1}{2}\left( -h_{1}e_{3}+h_{2}e_{4}\right) \text{ \  \ }h_{1}=-\left(
b+c\right) \text{ and }h_{2}=d
\end{equation}

By using $\left( 4\right) ,$ $\left( 9\right) $ and straight-forward
computation, the Laplacian $\Delta G$ of the Gauss map $G$ can be expressed
as%
\begin{equation}
\Delta G=A(s)\left( e_{1}\wedge e_{2}\right) +B(s)\left( e_{2}\wedge
e_{3}\right) +D(s)\left( e_{2}\wedge e_{4}\right)
\end{equation}%
where
\begin{equation}
A(s)=d^{2}\left( s\right) -b^{2}\left( s\right) -c^{2}\left( s\right)
\end{equation}%
\begin{equation}
B(s)=b^{\prime }\left( s\right) +c^{\prime }\left( s\right) +a(s)d^{2}\left(
s\right)
\end{equation}%
\begin{equation}
D(s)=d^{\prime }\left( s\right) +a(s)d\left( s\right) \left( b\left(
s\right) +c\left( s\right) \right)
\end{equation}

\begin{theorem}
\label{teo 1}Let $M$ be the flat rotation surface given by the
parametrization (5). Then $M$ has pointwise 1-type Gauss map if and only if
the profile curve of $M$ is parametrized by%
\begin{eqnarray}
\alpha _{1}(s) &=&a_{1} \\
\alpha _{3}(s) &=&\frac{1}{a_{2}}\left( 1+a_{1}^{2}\right) ^{\frac{1}{2}%
}\cos \left( a_{2}s+a_{3}\right)  \notag \\
\alpha _{4}(s) &=&-\frac{1}{a_{2}}\left( 1+a_{1}^{2}\right) ^{\frac{1}{2}%
}\sin \left( a_{2}s+a_{3}\right)  \notag
\end{eqnarray}

or%
\begin{eqnarray}
\alpha _{1}(s) &=&b_{1}s+b_{2} \\
\alpha _{3}(s) &=&\int \left( 1+b_{1}^{2}\right) ^{\frac{1}{2}}\cos \left(
b\ln \left \vert b_{1}s+b_{2}\right \vert \right) ds  \notag \\
\alpha _{4}(s) &=&\int \left( 1+b_{1}^{2}\right) ^{\frac{1}{2}}\sin \left(
b\ln \left \vert b_{1}s+b_{2}\right \vert \right) ds  \notag
\end{eqnarray}%
where $a_{1},$\ $a_{2},$\ $a_{3}$ $b_{1}\neq 0,$\ $b_{2}$, $b_{3}$ and $b=%
\frac{b_{3}}{b_{1}\left( 1+b_{1}^{2}\right) ^{\frac{1}{2}}}$ are real
constants.
\end{theorem}

\begin{proof}
Let $M$ be the flat rotation surface given by the parametrization (5). We
suppose that $M$ has pointwise 1-type Gauss map. By using (1) and (16), we
have
\begin{eqnarray}
f+f\left \langle C,e_{1}\wedge e_{2}\right \rangle &=&A(s) \\
f\left \langle C,e_{2}\wedge e_{3}\right \rangle &=&-B(s)  \notag \\
f\left \langle C,e_{2}\wedge e_{4}\right \rangle &=&D(s)  \notag
\end{eqnarray}%
and
\begin{equation}
\left \langle C,e_{1}\wedge e_{3}\right \rangle =\left \langle C,e_{1}\wedge
e_{4}\right \rangle =\left \langle C,e_{3}\wedge e_{4}\right \rangle =0
\end{equation}%
By differentiating (23) covariantly with respect to $s,$ we have
\begin{eqnarray*}
-a(s)B(s)+A(s)-f &=&0 \\
a(s)D(s) &=&0 \\
D(s) &=&0
\end{eqnarray*}%
In this case, firstly, we assume that $a(s)=0$ and $D(s)=0.$ From (10), we
obtain that $\alpha _{1}(s)=a_{1}.$ Since the profile curve is unit speed
spacelike curve, we can write $\left( \alpha _{3}^{\prime }(s)\right)
^{2}+\left( \alpha _{4}^{\prime }(s)\right) ^{2}=1+a_{1}^{2}.$ Also we can
put%
\begin{eqnarray}
\alpha _{3}^{\prime }(s) &=&\left( 1+a_{1}^{2}\right) ^{\frac{1}{2}}\cos
\theta \left( s\right) \\
\alpha _{4}^{\prime }(s) &=&\left( 1+a_{1}^{2}\right) ^{\frac{1}{2}}\sin
\theta \left( s\right)  \notag
\end{eqnarray}%
where $\theta $ is smooth angle function. On the other hand, since $D(s)=0,$
from (19) we obtain as
\begin{equation}
d\left( s\right) =a_{2},\text{ \  \  \  \ }a_{2}\text{\ is non zero constant.}
\end{equation}%
By using (13), (24) and (25) we get
\begin{equation}
\theta \left( s\right) =a_{2}s+a_{3}
\end{equation}%
So from (24) and (26) we have
\begin{eqnarray*}
\alpha _{3}(s) &=&\frac{1}{a_{2}}\left( 1+a_{1}^{2}\right) ^{\frac{1}{2}%
}\cos \left( a_{2}s+a_{3}\right) \\
\alpha _{4}(s) &=&-\frac{1}{a_{2}}\left( 1+a_{1}^{2}\right) ^{\frac{1}{2}%
}\sin \left( a_{2}s+a_{3}\right)
\end{eqnarray*}%
Now we assume that $a(s)\neq 0$ and $D(s)=0.$ Since $M$ is flat, (12) and
(14) imply that
\begin{equation}
\alpha _{1}(s)=b_{1}s+b_{2}
\end{equation}%
for some constants $b_{1}\neq 0$ and $b_{2}=0.$ Since the profile curve is
unit speed spacelike curve, we can write $\left( \alpha _{3}^{\prime
}(s)\right) ^{2}+\left( \alpha _{4}^{\prime }(s)\right) ^{2}=1+b_{1}^{2}.$
Also we can put%
\begin{eqnarray}
\alpha _{3}^{\prime }(s) &=&\left( 1+b_{1}^{2}\right) ^{\frac{1}{2}}\cos
\theta \left( s\right) \\
\alpha _{4}^{\prime }(s) &=&\left( 1+b_{1}^{2}\right) ^{\frac{1}{2}}\sin
\theta \left( s\right)  \notag
\end{eqnarray}%
where $\theta $ is smooth angle function. By using (10), (11) and $(19),$ we
get%
\begin{equation}
d(s)=\frac{b_{3}}{b_{1}s+b_{2}}
\end{equation}%
On the other hand, by using (13), (27) and (28) we have
\begin{equation}
d(s)=\left( 1+b_{1}^{2}\right) ^{\frac{1}{2}}\theta ^{\prime }\left( s\right)
\end{equation}%
By combining (29) and (30) we obtain
\begin{equation}
\theta \left( s\right) =b\ln \left \vert b_{1}s+b_{2}\right \vert
\end{equation}%
where $b=\frac{b_{3}}{b_{1}\left( 1+b_{1}^{2}\right) ^{\frac{1}{2}}}.$ So by
substituting (31) into (28) we can write
\begin{eqnarray}
\alpha _{3}(s) &=&\int \left( 1+b_{1}^{2}\right) ^{\frac{1}{2}}\cos \left(
b\ln \left \vert b_{1}s+b_{2}\right \vert \right) ds  \notag \\
\alpha _{4}(s) &=&\int \left( 1+b_{1}^{2}\right) ^{\frac{1}{2}}\sin \left(
b\ln \left \vert b_{1}s+b_{2}\right \vert \right) ds  \notag
\end{eqnarray}

Conversely, the surface $M$ which is parametrized by (20) and (21) is
pointwise 1-type Gauss map for
\begin{equation*}
f(s)=-a(s)b^{\prime }(s)-a^{2}(s)d^{2}(s)+d^{2}(s)-b^{2}(s)
\end{equation*}%
and
\begin{equation*}
C(s)=\frac{a(s)b^{\prime }(s)+a^{2}(s)d^{2}(s)}{f(s)}\left( e_{1}\wedge
e_{2}\right) +\frac{b^{\prime }(s)+a(s)d^{2}(s)}{f(s)}\left( e_{2}\wedge
e_{3}\right)
\end{equation*}%
where it can be easily seen that $e_{1}\left( C(s)\right) =0$ and $%
e_{2}\left( C(s)\right) =0.$ This completes the proof.
\end{proof}

\begin{corollary}
\label{cor 1}Let $M$ be the flat rotation surface given by the
parametrization (5). If $M$ has pointwise 1-type Gauss map then the profile
curve of $M$ is a helix curve.
\end{corollary}

We will also use the following theorems and corollary.

\begin{theorem}
\label{teo 2} \cite{dursun6} Let M be an oriented maximal surface in the
Minkowski space $E_{1}^{4}.$ Then $M$ has pointwise 1-type Gauss map of the
first kind if and only if $M$ has flat normal bundle. Hence the Gauss map $G$
satisfies (1.1) for $f=\left \Vert h\right \Vert ^{2}$ and $C=0.$
\end{theorem}

\begin{theorem}
\label{teo 3} \cite{ortega1} Let M be a spacelike rotational surface in
Minkowski 4-space given by the parametrization (5). If $M$ marginally
trapped surface then
\begin{eqnarray}
\alpha _{3}(s) &=&\int \left( 1+\left( \alpha _{1}^{\prime }\right)
^{2}\right) ^{\frac{1}{2}}\cos \theta \left( s\right) ds \\
\alpha _{4}(s) &=&\int \left( 1+\left( \alpha _{1}^{\prime }\right)
^{2}\right) ^{\frac{1}{2}}\sin \theta \left( s\right) ds  \notag
\end{eqnarray}%
and%
\begin{equation}
\theta \left( s\right) =-\epsilon \int \frac{1+\left( \alpha _{1}^{\prime
}\right) ^{2}+\alpha _{1}^{\prime }\alpha _{1}^{\prime \prime }}{\alpha
_{1}\left( 1+\left( \alpha _{1}^{\prime }\right) ^{2}\right) ^{\frac{1}{2}}}
\end{equation}%
where $\epsilon =\pm .$
\end{theorem}

\begin{corollary}
\label{cor 2}\cite{ortega1} Let M be a spacelike rotational surface in
Minkowski 4-space given by the parametrization (5). If $M$ is a extremal
surface then a unit profile curve is given by
\begin{equation*}
\alpha \left( s\right) =\left( f(s),0,\cos \zeta _{0}\sqrt{a_{1}}\arctan
\left( \frac{s+a_{2}}{f(s)}\right) ,\sin \zeta _{0}\sqrt{a_{1}}\arctan
\left( \frac{s+a_{2}}{f(s)}\right) \right) ,
\end{equation*}%
where $f(s)=\sqrt{a_{1}-\left( s+a_{2}\right) ^{2}}$and $a_{1},a_{2},\zeta
_{0}\in \mathbb{R},$ $a_{1}>0,$ being integration constants. In particular,
the surface $M$ is immersed in a totally geodesic Lorentzian 3-space.
\end{corollary}

\begin{theorem}
\label{teo 4}Let $M$ be the marginally trapped surface given by the
parametrization (5) in Minkowski 4-space. Then $M$ has pointwise 1-type
Gauss map if and only if the profile curve is given by or%
\begin{eqnarray}
\alpha _{1}(s) &=&\left( \lambda _{1}-1\right) ^{\frac{1}{2}}\left(
u^{2}\left( s\right) +\lambda ^{2}\right) ^{\frac{1}{2}} \\
\alpha _{3}(s) &=&\int \left( \frac{\lambda _{1}u^{2}+\lambda ^{2}}{%
u^{2}+\lambda ^{2}}\right) ^{\frac{1}{2}}\cos \theta \left( s\right) ds
\notag \\
\alpha _{4}(s) &=&\int \left( \frac{\lambda _{1}u^{2}+\lambda ^{2}}{%
u^{2}+\lambda ^{2}}\right) ^{\frac{1}{2}}\sin \theta \left( s\right) ds
\notag
\end{eqnarray}%
and
\begin{equation*}
\theta \left( s\right) =-\epsilon \frac{\lambda _{1}}{\left( \lambda
_{1}-1\right) ^{\frac{1}{2}}}\int \frac{\left( u^{2}+\lambda ^{2}\right) ^{%
\frac{1}{2}}}{\lambda _{1}u^{2}+\lambda ^{2}}ds
\end{equation*}%
where $u\left( s\right) =\delta s+\lambda _{3},$ $\lambda =\frac{\lambda _{2}%
}{\lambda _{1}-1},$ $\lambda _{1},$\ $\lambda _{2}$, $\lambda _{3},$ $a_{1}$
and $a_{2}$ are real constants.
\end{theorem}

\begin{proof}
Let $M$ be marginally trapped surface. This means $\left \Vert H\right \Vert
=0 $ that is $\left \langle H,H\right \rangle =0.$ By using (15), we get
\begin{equation}
-\left( b(s)+c(s)\right) =\epsilon d(s)
\end{equation}%
where $\epsilon =\pm .$ In this case, by using (35) we can rewrite the
Laplacian $\Delta G$ of the Gauss map $G$ as%
\begin{equation}
\Delta G=A(s)\left( e_{1}\wedge e_{2}\right) -\epsilon N(s)\left(
e_{2}\wedge e_{3}\right) +N(s)\left( e_{2}\wedge e_{4}\right)
\end{equation}%
where
\begin{equation}
N(s)=d^{\prime }\left( s\right) -\epsilon a(s)d^{2}\left( s\right)
\end{equation}%
We assume that $M$ has pointwise 1-type Gauss map. Then we have
\begin{eqnarray}
f+f\left \langle C,e_{1}\wedge e_{2}\right \rangle &=&A(s) \\
f\left \langle C,e_{2}\wedge e_{3}\right \rangle &=&\epsilon N(s)  \notag \\
f\left \langle C,e_{2}\wedge e_{4}\right \rangle &=&N(s)  \notag
\end{eqnarray}%
and
\begin{equation}
\left \langle C,e_{1}\wedge e_{3}\right \rangle =\left \langle C,e_{1}\wedge
e_{4}\right \rangle =\left \langle C,e_{3}\wedge e_{4}\right \rangle =0
\end{equation}%
By differentiating (39) covariantly with respect to $s,$ we have
\begin{eqnarray*}
\epsilon a(s)N(s)+A(s)-f &=&0 \\
a(s)N(s) &=&0 \\
N(s) &=&0
\end{eqnarray*}%
In this case, firstly, we assume that $a(s)=0$ and $N(s)=0.$ From (10) and
(12), we obtain that $\alpha _{1}(s)=a_{1}$ and $c(s)=0,$ respectively.
Hence from (35) we get
\begin{equation}
-b(s)=\epsilon d(s)
\end{equation}%
By using $(40)$ and (17) we obtain that $A(s)=0$. So we have that $f=0.$
This is a contradiction.

Now we assume that $a(s)\neq 0$ and $N(s)=0.$ By combining (10), (11), (12),
(13), (35) and (37), we obtain a differential equation as follows:%
\begin{equation*}
\left( 1+\left( \alpha _{1}^{\prime }(s)\right) ^{2}+\alpha _{1}^{\prime
}(s)\alpha _{1}^{\prime \prime }(s)\right) ^{\prime }\alpha _{1}(s)\left(
1+\left( \alpha _{1}^{\prime }(s)\right) ^{2}\right) =0
\end{equation*}%
Since $\alpha _{1}>0$ and $1+\left( \alpha _{1}^{\prime }(s)\right) ^{2}\neq
0$ we have%
\begin{equation*}
1+\left( \alpha _{1}^{\prime }(s)\right) ^{2}+\alpha _{1}^{\prime }(s)\alpha
_{1}^{\prime \prime }(s)=\lambda _{1}
\end{equation*}%
whose the solition
\begin{equation}
\alpha _{1}(s)=\left( \lambda _{1}-1\right) ^{\frac{1}{2}}\left( \left(
\delta s+\lambda _{3}\right) ^{2}+\frac{\lambda _{2}}{\left( \lambda
_{1}-1\right) ^{2}}\right) ^{\frac{1}{2}}
\end{equation}%
By using (33) and (41) we get
\begin{equation}
\theta \left( s\right) =-\epsilon \mu \int \frac{\left( u^{2}+\lambda
\right) ^{\frac{1}{2}}}{\lambda _{1}u^{2}+\lambda }ds
\end{equation}%
where $u(s)=\delta s+\lambda _{3}$, $\lambda =\frac{\lambda _{2}}{\lambda
_{1}-1}$ and $\mu =\frac{\lambda _{1}}{\left( \lambda _{1}-1\right) ^{\frac{1%
}{2}}}$.

Conversely, the surface $M$ which is parametrized\ by (34) has pointwise
1-type Gauss map with%
\begin{equation*}
f(s)=2b\left( s\right) c\left( s\right)
\end{equation*}%
and
\begin{equation*}
C(s)=0
\end{equation*}%
This completes the proof.
\end{proof}

\begin{corollary}
\label{cor 3}Let $M$ be marginally trapped surface given by the
parametrization (5) in Minkowski 4-space. Then $M$ has pointwise 1-type
Gauss map then M is pointwise 1-type Gauss map of the first kind.
\end{corollary}

\begin{corollary}
\label{cor 4}Let $M$ be a spacelike rotational surface in Minkowski 4-space
given by the parametrization (5). If $M$ is extremal surface then $M$ has
pointwise 1-type Gauss map of the first kind.
\end{corollary}

\begin{proof}
We assume that $M$ is a spacelike rotational surface given by the
parametrization (5). In that case by using (3) and (7) we obtain that $M$
has flat normal bundle. Hence from Theorem ($\ref{teo 2}$) If $M$ is
extremal surface then $M$ has pointwise 1-type Gauss map of the first kind.
\end{proof}

\end{document}